\documentclass[12pt]{amsart}
\usepackage{fullpage, amssymb, stmaryrd}
\usepackage[all]{xy}
\usepackage{url}
\usepackage{color}
\usepackage{soul,xcolor} %erase
\usepackage{comment}
\usepackage{tikz-cd}
\usetikzlibrary{cd}
\setstcolor{red}

\newtheorem{theorem}{Theorem}[section]
\newtheorem{prop}[theorem]{Proposition}
\newtheorem{lemma}[theorem]{Lemma}
\newtheorem{cor}[theorem]{Corollary}
\newtheorem{problem}[theorem]{Problem}
\newtheorem{conj}[theorem]{Conjecture}
\theoremstyle{definition}
\newtheorem{defn}[theorem]{Definition}
\newtheorem{remark}[theorem]{Remark}
\newtheorem{example}[theorem]{Example}

\numberwithin{equation}{theorem}

\newcommand{\frakm}{\mathfrak{m}}
\newcommand{\calE}{\mathcal{E}}
\newcommand{\calG}{\mathcal{G}}
\newcommand{\calL}{\mathcal{L}}
\newcommand{\calS}{\mathcal{S}}
\newcommand{\calO}{\mathcal{O}}

\newcommand{\AAA}{\mathbb{A}}
\newcommand{\CC}{\mathbb{C}}
\newcommand{\FF}{\mathbb{F}}
\newcommand{\GG}{\mathbb{G}}
\newcommand{\PP}{\mathbb{P}}
\newcommand{\QQ}{\mathbb{Q}}
\newcommand{\ZZ}{\mathbb{Z}}

\DeclareMathOperator{\Br}{Br}

\DeclareMathOperator{\et}{et}
\DeclareMathOperator{\fppf}{fppf}
\DeclareMathOperator{\PGL}{PGL}
\DeclareMathOperator{\Pic}{Pic}
\DeclareMathOperator{\Gal}{Gal}

\DeclareMathOperator{\Res}{Res}
\DeclareMathOperator{\Spec}{Spec}
\DeclareMathOperator{\SY}{SY}

\DeclareMathOperator{\Hom}{Hom}

\begin{document}

\title{Tamely ramified morphisms of curves and Belyi's theorem in positive characteristic}
\author{Kiran S. Kedlaya, Daniel Litt, and Jakub Witaszek}
\address{Department of Mathematics, University of California, San Diego, La Jolla, CA 92093, USA}
\email{kedlaya@ucsd.edu}
\address{Department of Mathematics, University of Georgia, Athens, GA 30602, USA}
\email{dlitt@uga.edu}
\address{Department of Mathematics, University of Michigan, Ann Arbor, MI 48109, USA}
\email{jakubw@umich.edu}

\date{\today}
\thanks{Thanks to Seichi Yasuda for directing us to Hoshi's work on the Schwarzian derivative and to Piotr Achinger and Maciej Zdanowicz for helpful comments. {We thank the referee for their suggestions and for reading the article carefully.} This paper began while all three authors were at IAS for the 2018--2019 academic year,
supported by a visiting professorship (Kedlaya) and NSF grant DMS-1638352 (Litt, Witaszek).
Kedlaya and Witaszek also visited MSRI (NSF grant DMS-1440140) during spring 2019.
Kedlaya was additionally supported by NSF (grants DMS-1501214, DMS-1802161) and UCSD (Warschawski Professorship). Litt was additionally supported by NSF grant DMS-2001196.}

\begin{abstract}
We show that every smooth projective curve over a finite field $k$ admits a finite tame morphism to the projective line over $k$. Furthermore, we construct a curve with no such map when $k$ is an infinite perfect field of characteristic two. Our work leads to a refinement of the tame Belyi theorem in positive characteristic, building on results of Sa\"idi, Sugiyama--Yasuda, and Anbar--Tutdere.
\end{abstract}

\maketitle

\section{Introduction}

The difference between tame and wild ramification plays a crucial role in positive characteristic algebraic geometry. Roughly speaking, finite tame morphisms behave as if the characteristic was zero while wild morphisms do not.
For example, the Riemann-Hurwitz formula for a finite map $f \colon X \to Y$ of smooth projective curves holds in characteristic $p>0$ without modification when $f$ is tame,
but must be modified somewhat when $f$ is wild.
For another example, the study of fundamental groups of curves over a base field of characteristic 0
is largely mirrored by Grothendieck's theory of tame \'etale fundamental groups \cite[Expos\'e~X]{Gro71}.
It is thus a fundamental question whether or not a given curve admits a tame morphism to $\PP^1_k$.

In most characteristics, one can produce tame morphisms by considering morphisms {whose ramification indices are small}.
To wit, we say that $f$ is \emph{simply ramified} if the ramification indices are equal to one or two.
The following result is classical for $p=0$, and due to Fulton for $p \neq 0$
 \cite{Ful69}.

\begin{theorem}[Fulton] \label{T:fulton}
{\noindent Let $X$ be a smooth, projective, geometrically irreducible curve over a field $k$.}
If $p=0$, or if $p>2$ and $k$ is infinite, then there exists a finite separable morphism $f: X \to \PP^1_k$ which is everywhere simply ramified,
and hence tame.
\end{theorem}

The main goal of this paper is to address the two missing cases in this theorem. One of these is where $p>2$ and $k$ is finite; in this case, one can extend the result of Fulton by a sieving argument (Theorem~\ref{T:odd finite})
in the style of the Bertini theorem of Poonen \cite{Poo04}. 

The remaining, more substantive case is when $p=2$, as then simply ramified morphisms are not always tame.
In fact, it was previously known that $X$ need not itself admit a tame separable morphism to $\PP^1_k$ when $k$ is not perfect,
by an example of Schr\"oer (see Example~\ref{exa:schroer}); on the other hand, a recent breakthrough result of Sugiyama--Yasuda \cite{SY18} shows that such a morphism always exists if $k$ is algebraically closed. In particular, the obstruction to such a morphism is arithmetic (depending on $k$) rather than geometric.
 
Using the Sugiyama--Yasuda construction, we prove the following.

\begin{theorem} \label{T:main1}
If $k$ is finite, then there exists a finite separable tame morphism $f: X \to \PP^1_k$.
\end{theorem}

\begin{theorem} \label{T:main2}
There always exists a finite extension $k'$ of $k$, of degree depending only on the genus of $X$ {(except possibly if this genus is $1$)}, for which
$X_{k'}$ admits a {finite separable tame morphism} to $\PP^1_{k'}$. 
\end{theorem}

We establish Theorem~\ref{T:main1} and Theorem~\ref{T:main2} by analyzing the Sugiyama--Yasuda construction in geometric terms: it gives rise to a canonical collection of smooth conic bundles over $X$ with the property that the existence of a tame morphism from $X$ to $\PP^1_k$ is equivalent to the triviality of some conic bundle in the collection. (When $X$ is ordinary, this collection is reduced to a single bundle. In the general case, we do not know whether different bundles in the collection represent the same Brauer class.) 

\begin{comment}
Using the Sugiyama--Yasuda construction, we give two further results on tame morphisms in characteristic 2.
One of these is the following affirmative statement.

\begin{theorem} \label{T:main}
If $p=2$, then the following assertions hold.
\begin{enumerate}
\item[(a)]
There exists a finite extension $k'$ of $k$ for which
there exists a tame separable morphism from $X_{k'}$ to $\PP^1_{k'}$. 
\item[(b)]
If $k$ is finite, then we may take $k' = k$.
\item[(c)]
For general $k$, we may ensure that $[k':k]$ is bounded by some function of $g_X$ alone.
\end{enumerate}
\end{theorem}

Here part (a) is simply a restatement of the Sugiyama--Yasuda result. 
We establish (b) and (c) by analyzing the Sugiyama--Yasuda construction in geometric terms: it gives rise to a canonical collection of smooth conic bundles over $X$ with the property that the existence of a tame morphism from $X$ to $\PP^1_k$ is equivalent to the triviality of some conic bundle in the collection. (When $X$ is ordinary, this collection is reduced to a single bundle. In the general case, we do not know whether different bundles in the collection represent the same Brauer class.) 
\end{comment}

Our other result is in the negative direction. The example of Schr\"oer cited above shows that when $k$ is not perfect, the finite extension $k'/k$ in Theorem~\ref{T:main2} is sometimes forced to contain a certain nontrivial purely inseparable extension of $k$. However, using the explicit nature of the Sugiyama--Yasuda construction, 
one can show that even when $k$ is perfect, one can find examples where $k'$ cannot equal $k$, as in the following
case (see \S\ref{sec:elliptic curves}).

\begin{theorem}
Let $k$ be the perfect closure of $\FF_2(t)$. Then there exists an ordinary elliptic curve $X$ over $k$
for which there exists no {finite separable tame morphism} from $X$ to $\PP^1_k$.
\end{theorem}

To conclude this introduction, we give an application of Theorem~\ref{T:main1} to a refinement of the tame Belyi theorem in positive characteristic. Recall that Belyi's theorem \cite{Bel80, Bel02}
is commonly asserted in the following form: for $p=0$, $X$ admits a finite morphism to $\PP^1_k$ ramified only over $\{0,1,\infty\}$ if and only if $X$ descends to $\overline{\QQ}$. In this statement, the ``only if'' assertion is essentially due to Riemann, while the ``if'' assertion is Belyi's contribution and can be made somewhat more precise.
To wit, suppose that $X$ is defined over a subfield $k$ of $\overline{\QQ}$. Then for any finite morphism $f_0: X \to \PP^1_k$, one can find a finite morphism $f_1: \PP^1_{k} \to \PP^1_{k}$ such that $f_1 \circ f_0$ is ramified only over $\{0,1,\infty\}$: the proofs of \cite[Theorem~2.2, Theorem~2.4]{Gol12} produce such an $f_1$ by base extension from $\QQ$.

In positive characteristic, the direct analogue of this statement fails spectacularly: regardless of how big $k$ is,
$X$ admits a morphism to $\PP^1_k$ ramified only over $\infty$ (see Theorem~\ref{T:wild}). If one restricts to tamely ramified morphisms, however, then the ``only if'' assertion of Belyi's theorem becomes true again thanks to Grothendieck's theory of tame \'etale fundamental groups, which implies that Riemann's rigidity property persists.
The ``if'' assertion was established for $k = \overline{\FF}_p$ with $p>2$ by Sa\"\i di \cite[Th\'eor\`eme~5.6]{Sai97}; the case $p=2$ was handled by Anbar--Tutdere \cite[Theorem~2]{AT18} using the Sugiyama--Yasuda result.

We extend to positive characteristic the refined version of Belyi's theorem. As this includes the existence of a tame morphism when $k$ is finite, this is new both for $p>2$ and for $p=2$. See Theorem~\ref{T:tame Belyi positive characteristic} for the proof. 
\begin{theorem} \label{T:tame Belyi}
Suppose that $p>0$ and $k$ is algebraic over $\FF_p$. Then there exists a finite separable tame morphism $f: X \to \PP^1_k$ ramified only over $\{0,1,\infty\}$ (without base extension from $k$ to $\overline{\FF}_p$).
\end{theorem}

\subsection*{Notations}

Throughout this paper, as above let $k$ be a field of characteristic $p \geq 0$;
 let $X$ be a smooth, projective, geometrically irreducible curve of genus $g_X$ over $k$;
 and let $X^\circ$ denote the set of closed points of $X$. Let $\overline{k}$ be an algebraic closure of $k$.
Let $R_X$ denote the set of finite separable nonconstant morphisms $f: X \to \PP^1_k$,
where the target is equipped with a fixed coordinate; by pulling back this coordinate, we identify
elements of $R_X$ with elements of the function field $k(X)$ with nonzero derivative.
(For $p=0$ this excludes only constants; for $p>0$ it excludes elements of the subfield $k \cdot k(X)^p$.)
For any field extension ${k'}$ of $k$, let $X_{{k'}}$ denote the base extension of $X$ from $k$ to ${k'}$,
and write $R_{X,{k'}}$ in place of $R_{X_{k'}}$.

For $f: X \to \PP^1_k$ a finite separable morphism 
and $x \in {X(\overline{k})}$, the \emph{ramification index} of $f$ at $x$ is the positive integer $e_x$ for which
\[
f^{-1} \frakm_{\PP^1_{\overline{k}}, f(x)} \cdot \calO_{X_{\overline{k}}} = \frakm_{X_{\overline{k}},x}^{e_x}.
\]
We say that $f$ is \emph{tamely ramified} at $x$ if $e_x$ is not divisible by $p$, and \emph{wildly ramified} at $x$ otherwise;
if $f$ is tamely ramified at every $x \in X$, we simply say that $f$ is \emph{tame}.
(For the purposes of checking this condition, it is permissible to work with points over a perfect closure of $k$ instead of an algebraic closure; this will be helpful later.)

\section{Odd characteristic}

We first treat the case of odd characteristic using a probabilistic argument in the style of Poonen's Bertini theorem over finite fields \cite{Poo04}.

\begin{theorem} \label{T:odd finite}
If $p>2$ and $k$ is finite, then there exists $f \in R_X$ which is everywhere simply ramified,
and hence tame.
\end{theorem}
\begin{proof}
Let $\calL$ be an ample line bundle on $X$ and let $n$ be a positive integer.
For $x \in X$, choose a generator $t_x$ of $\calL$ at $x$.
For any pair $(s_0, s_1) \in H^0(X, \calL^{\otimes n})^{\times 2}$, we can {choose a trivialization of $\calL$ in a neighborhood of $x$, thus identifying $t_x$ with a uniformizer in $\calO_{X,x}$, and then} expand around $x$ to obtain
\[
s_0 = s_{0,0} + s_{0,1} t_x + s_{0,2}t_x^2 + O(t_x^3), \qquad
s_1 = s_{1,0} + s_{1,1} t_x + s_{1,2}t_x^2 + O(t_x^3);
\]
the ratio $f = s_0/s_1 \in k(X)$ defines a morphism to $\PP^1_k$ in a neighborhood of $x$ if and only if
$s_{0,0}$ and $s_{1,0}$ are not both zero.
If this occurs, $f$ is ramified at $x$ if and only if $s_{0,0} s_{1,1} - s_{0,1} s_{1,0} = 0$.
If this also occurs, then $f$ fails to be simply ramified at $x$ if and only if $s_{0,0} s_{1,2} - s_{1,0} s_{0,2} = 0$. Here we used that $(\frac{s_0}{s_1})' = \frac{s'_0s_1 - s_0s'_1}{s_1^2}$ and
$(\frac{s_0}{s_1})'' = \frac{s''_0s_1 - s_0s''_1}{s_1^2}$.

Consequently, as soon as $n$ is large enough that $H^0(X, \calL^{\otimes n})$ surjects onto 
\[
H^0(\Spec \calO_{X,x}/\frakm_{X,x}^3, \calL^{\otimes n}) \cong \calO_{X,x}/\frakm_{X,x}^3,
\] 
{we may compute the probability that a random pair $(s_0, s_1)$ defines a simply ramified morphism at $x$ as follows. The number of $k$-rational points on the affine quadric in six variables 
\[
s_{0,0} s_{1,1} - s_{0,1} s_{1,0} = s_{0,0} s_{1,2} - s_{1,0} s_{0,2} = 0
\]
is $(\kappa(x)^2-1)\kappa(x)^2 + \kappa(x)^4 = 2\kappa(x)^4 - \kappa(x)^2$: for every $(s_{0,0}, s_{1,0})\neq (0,0)$ we have $\kappa(x)^2$ choices for the other variables, and for $(s_{0,0}, s_{1,0})= (0,0)$ we have $\kappa(x)^4$ such choices. Hence the desired probability is $\frac{\kappa(x)^6 - 2\kappa(x)^4 + \kappa(x)^2}{\kappa(x)^6} = (1 - \kappa(x)^{-2})^2$.}

Put $q = \# k$. Then for $m$ a positive integer, the number of $x \in X^\circ$ with $\#\kappa(x) = q^m$
is $O(q^m)$; it follows that the product $\prod_{x \in X^\circ} (1 - \kappa(x)^{-2})^2$ converges absolutely to a positive limit (namely $\zeta_X(2)^{-2}$, where $\zeta_X(s)$ is the zeta function of $X$). Consequently, if we write $S_n$ for the set of pairs $(s_0, s_1) \in H^0(X, \calL^{\otimes n})^{\times 2}$ which define a morphism in $R_X$ which is everywhere simply ramified, it will suffice to check that
\[
\lim_{n \to \infty} \frac{\#S_n}{\#H^0(X, \calL^{\otimes n})^{\times 2}} = 
\prod_{x \in X^\circ} (1 - \kappa(x)^{-2})^2,
\]
as this will then imply that $S_n \neq \emptyset$ for $n$ large. Note that this does not follow from the previous paragraph, because the number $n$ for which $H^0(X, \calL^{\otimes n})$ surjects onto $\calO_{X,x}/\frakm_{X,x}^3$ depends on the point $x$. To circumvent this problem we follow the paradigm of \cite{Poo04}.

Fix a positive integer $e$. We then distinguish points of $x$ as being of \emph{low degree}, \emph{medium degree}, or \emph{high degree} according to whether the degree of $x$ over $k$ belongs to 
\[
[1, e], \qquad \left[e+1,  \tfrac{n}{2} \right], \qquad
\left( \tfrac{n}{2} , \infty \right).
\]
For $n$ large compared to $e$, the preceding analysis shows that $(s_1,s_2)$ define a simply ramified morphism at each point of low degree with probability equal to the product of $(1 - \kappa(x)^{-2})^2$ as $x$ ranges over these points.
For points of medium degree, we may apply {\cite[Lemma~2.5]{BK12}} to see
that the probability that the morphism is ramified (simply or not) at some such point tends to 0 as $e \to \infty$
(uniformly in $n$).
For points of high degree, we may apply {\cite[Lemma~2.6]{BK12}} to see that 
the probability that the morphism is ramified (simply or not) at some such point tends to 0 as $n \to \infty$.
Combining these results prove the claim.
\end{proof}

\begin{remark}
We did not see how to deduce Theorem~\ref{T:odd finite} directly from the main results of \cite{Poo04}. It would be interesting to see whether one of the many subsequent variations of that result give a direct implication of
Theorem~\ref{T:odd finite}, or if not whether there is a hitherto unknown variation that would do so.
\end{remark}

\section{An obstruction to tame morphisms}

We now assume $p=2$ until further notice.
In this context, we next describe an obstruction to the existence of tame morphisms
discovered by Schr\"oer \cite[Theorem~6.1]{Sch03a}.

\begin{prop}[Schr\"oer] \label{P:tame to square}
Suppose that there exists $f \in R_X$ which is tame. Then 
the canonical bundle $\omega_{X/k}$ is a square in $\Pic(X)$.
\end{prop}
\begin{proof}
Write $\omega_{X/k}$ as the tensor product of the pullback $f^* \omega_{\PP^1_k/k}$ with
the determinant $\calO_X(K_{X/\PP^1})$ of the relative canonical sheaf. 
The former is a square because $\omega_{\PP^1_k/k} \cong \calO_{\PP_k^1}(-2)$;
since $f$ is tame, the latter is the square of $\calO_{X}(D)$ for (cf.\ \cite[Tag 0C1F]{SP})
\[
D = \sum_{x \in X^\circ} \tfrac{e_x - 1}{2} [x].
\]
Hence $\omega_X$ is itself a square.
\end{proof}

\begin{example} \label{exa:schroer}
Let $X$ be the generic $n$-pointed curve of genus $g$ for some integers $n \geq 0, g \geq 3$.
Then a theorem of Schr\"oer \cite[Theorem~5.1]{Sch03a} (the ``strong Franchetta conjecture'')
implies that $\omega_{X/k}$ is not a square in $\Pic(X)$. Consequently, Proposition~\ref{P:tame to square} implies that
$X$ admits no tame morphism to $\PP^1_k$.

By contrast, it was shown by Bouw--Wewers \cite[Theorem~1]{BW05} that (for arbitrary $p$) for $n \geq 3, g \geq 0$,
the generic $n$-pointed genus-$g$ curve does admit a tame morphism after base extension to an algebraically closed field. This statement is now subsumed by the result of Sugiyama--Yasuda, except that Bouw--Wewers obtain some additional control on the degree of the morphism and on the ramification indices; see \cite[Theorem~10]{BW05}. For example, for $p \neq 3$ one can ensure that the ramification indices are all equal to 1 or 3.
\end{example}

We record a reformulation of the condition that the canonical bundle is a square.
\begin{defn}
For $i$ a nonnegative integer and $U$ an open subset of $X$, let $U^{(i)}$ be the base extension of $U$ along the $i$-th power of the absolute Frobenius morphism on $k$.
Recall that the $i$-th power of the absolute Frobenius morphism on $X$ factors as the relative Frobenius
$\pi^{(i)}: X \to X^{(i)}$ followed by the base change morphism $X^{(i)} \to X$.
\end{defn}

\begin{defn}
For $p$ arbitrary, 
a \emph{theta characteristic} {(also known as a \emph{spin structure})} on $X$ is a line bundle $\calL$ on $X$ for which $\calL^{\otimes 2} \cong \omega_{X/k}$.
If such a bundle exists, then the set of isomorphism classes of theta characteristics form a torsor for $\Pic(X)[2]$.
		
For $p=2$, there exists a \emph{canonical theta characteristic} over $X^{(1)}$;
as in \cite{SV87}, it may be constructed by observing that for any $f \in k(X)$,
the divisor of $df$ becomes a square over $X^{(1)}$. Moreover, it is unique because for any $g \in R_X$ the ratio $\frac{df}{dg}$ is a square in $k(X)$.
\end{defn}

\begin{remark}
From the description of the canonical theta characteristic, it
is clear that the obstruction to tame morphisms described in Proposition~\ref{P:tame to square} 
vanishes upon base extension along Frobenius on $k$.
In particular, this obstruction vanishes whenever $k$ is perfect.
\end{remark}

\section{The Sugiyama-Yasuda symbol}
\label{sec:sugiyama-yasuda}

The breakthrough of Sugiyama--Yasuda rests on the remarkable invariant theory for the group $\PGL_2(k^{1/4} \cdot k(X))$ developed in \cite[\S 2]{SY18}, which will make it possible to study ``tameness modulo fourth powers''
starting in \S\ref{sec:pseudotame morphisms}. 
This can be viewed as a characteristic-2 analogue of the classical theory of the Schwarzian derivative
(see Remark~\ref{R:symbol}).
We give a detailed treatment here, both for expository purposes and to clarify the effect of working over an arbitrary (not necessarily algebraically closed or even perfect) base field. {Throughout this section we assume that $\mathrm{char}(k)=2$.}

\begin{defn}
Put $\Gamma := \PGL_2(k^{1/4} \cdot k(X))$. 
We write $\Gamma_k$ in place of $\Gamma$ when it becomes necessary to specify $k$ (e.g., when passing to a field extension).

Consider the following action of $\Gamma$ on $k(X)$ by linear fractional transformations:
\[
\begin{pmatrix} a & b \\
c & d \end{pmatrix}(f) = \frac{a^4 f + b^4}{c^4 f + d^4}.
\]
Note that the action on $R_X \cong k(X) \setminus k \cdot k(X)^2$
is free: a fixed point would correspond to a solution of the equation $a^4 f + b^4 = c^4 f^2 + d^4 f$, but since $f \notin k \cdot k(X)^2$ this would force $b=c=0$ and $a=d$.

Write $\Gamma f$ for the $\Gamma$-orbit of $f$; note that it contains both $f^{-1}$ and $f^3 = f^4 f^{-1}$.
\end{defn}

\begin{defn} \label{D:symbol}
For $g \in R_X$, note that $1,g,g^2,g^3$ form a basis\footnote{{Since $K(X)^4 \subseteq K(X)^2 \subseteq K(X)$, it is enough to check that $1,g$ is a basis of $K(X)^2 \subseteq K(X)$ which is clear as it is an extension of degree two and $1,g$ are linearly independent which can be seen by differentiating any linear relation thereof.}} of $k(X)$ over $k \cdot k(X)^4$ .
For $f,g \in R_X$, we may therefore write 
\begin{equation} \label{eq:f in terms of g}
f = f_0^4 + f_1^4 g + f_2^4g^2 + f_3^4 g^3 = (f_0^2 + f_2^2 g)^2 + (f_1^2 + f_3^2 g)^2 g \qquad (f_i \in k^{1/4} \cdot k(X)).
\end{equation}
Since $f_1^2 + f_3^2 g = \left( \frac{df}{dg} \right)^{1/2}$ is nonzero (as otherwise $f$ would lie in $k \cdot k(X)^2$), we may define the \emph{Sugiyama-Yasuda symbol} of $f$ and $g$ by the formula
\[
\SY(f,g) := \left( \frac{f_1 f_3 + f_2^2}{f_1^2 + f_3^2 g} \right)^2 \,dg \in \Omega_{k(X)/k}.
\]
(Note that $\SY$ is meant to abbreviate both ``Sugiyama-Yasuda'' and ``symbol''; see Remark~\ref{R:symbol}.)
\end{defn}

\begin{remark} \label{R:f in terms of g transformed}
The group $\Gamma$ is generated by the operations
\[
 g \mapsto g+1, \qquad g \mapsto g^{-1}, \qquad g \mapsto t^4g \qquad (t \in (k^{1/4} \cdot k(X))^\times).
\]
These operations have the following effects on \eqref{eq:f in terms of g}:
\begin{align*}
f &= (f_0 + f_1 + f_2 + f_3)^4 + (f_1 + f_3)^4 (g+1)  + (f_2 + f_3)^4 (g+1)^2 + f_3^4 (g+1)^3\\
&= f_0^4 + (f_3 g)^4 g^{-1} + (f_2 g)^4 g^{-2} + (f_1 g)^4 g^{-3} \\
&= f_0^4 + (t^{-1} f_1)^4 t^4 g + (t^{-2} f_2)^4 (t^4 g)^2 + (t^{-3} f_3)^4 (t^4 g)^3.
\end{align*}
In particular, by replacing $g$ with a suitable element in $\Gamma g$, we can first ensure that $f_3 \neq 0$
(as otherwise the nonvanishing of $f_1^2 + f_3^2 g$ forces $f_1 \neq 0$, and we may apply $g \mapsto g^{-1}$),
and then that $f_2 = 0$ (by first rescaling to achieve $f_2 = f_3$).
If we further replace $f$ with $f + f_0^4$ (thus
moving $f$ within $\Gamma f$), we then have $f/g = (f_1^2 + f_3^2 g)^2$ and so
\[
\SY(f,g) = \frac{(f_1 f_3)^2}{f/g}\,dg.
\]
\end{remark}

\begin{lemma} \label{L:symbol is invariant}
For $f,g \in R_X$, {we have that $\SY(f,g) = \SY(f, \gamma(g))$ for every $\gamma \in \Gamma$}.
\end{lemma}
\begin{proof}
By Remark~\ref{R:f in terms of g transformed},
\begin{align*}
\SY(f,t^4g) &= \left( \frac{t^{-1} f_1 t^{-3} f_3 + t^{-4} f_2^2}{t^{-2} f_1^2 + t^{-6} f_3^{2} t^4 g} \right)^2 d(t^4 g) = \SY(f,g) \\
\SY(f,g+1) &= \left( \frac{(f_1 + f_3)f_3 + (f_2+f_3)^2}{(f_1 + f_3)^2 + f_3^2(g+1)} \right)^2 d(g+1) = \SY(f,g) \\
\SY(f,g^{-1}) &= \left( \frac{f_1 g f_3 g + f_2^2 g^2}{f_3^2 g^2 + f_1^2 g^2 g^{-1}} \right)^2 d(g^{-1}) = \SY(f,g).
\end{align*}
This establishes invariance under a generating set of $\Gamma$ and hence proves the claim.
\end{proof}

\begin{cor} \label{C:symbol is zero}
For $f,g \in R_X$, $\SY(f,g) = 0$ if and only if $g \in \Gamma f$.
\end{cor}
\begin{proof}
For $f=g$, we have $f_0 = f_2 = f_3 = 0$ and so $\SY(f,g) = 0$. By Lemma~\ref{L:symbol is invariant}, 
it follows that if $g \in \Gamma f$, then $\SY(f,g) = 0$.

Conversely, suppose that $\SY(f,g) = 0$ and we wish to check that $f \in \Gamma g$. 
By Lemma~\ref{L:symbol is invariant}, both the hypothesis and the conclusion are preserved by moving $g$ within
$\Gamma g$, so by Remark~\ref{R:f in terms of g transformed} we may assume that $f_2 = 0, f_3 \neq 0$.
The vanishing of $\SY(f,g)$ implies that $f_1 f_3 + f_2^2 = 0$; we must then have $f_1 = 0$ in addition, and
so $f = f_0^4 + f_3^4 g^3 \in \Gamma g^3 = \Gamma g$.
\end{proof}

We come now to the most remarkable property of $\SY$.
\begin{lemma} \label{L:symbol is additive}
For $f,g,h \in R_X$,
\[
\SY(f,g) + \SY(g,h) = \SY(f,h).
\]
\end{lemma}
\begin{proof}
When computing $\SY(f,g)$, Lemma~\ref{L:symbol is invariant}
asserts that we are free to move $g$ within $\Gamma g$, and Remark~\ref{R:f in terms of g transformed}
asserts that by so doing (and translating $f$ by a fourth power) we can ensure that $f \in k \cdot k(X)^2 g$. 
By the same logic, we may
move $h$ within $\Gamma h$ while fixing $f$ and $g$, so as to ensure that $g \in k \cdot k(X)^4 + k \cdot k(X)^2 h$ (however, we cannot move $g$ without disturbing our assumption about $f$).
Writing $r$ for $g_1^2 + g_3^2 h = (dg/dh)^{1/2}$, we have
\begin{align*}
g &= g_0^4 + r^2 h = g_0^4 + (g_1^2 + g_3^2 h)^2 h \\
f &= (f_1^2 + f_3^2 g)^2 g \\
&= (f_1^2 + f_3^2 (g_0^4 + r^2 h))^2 (g_0^4 + g_1^4 h + g_3^4 h^3) \\
&= ((f_1 + f_3 g_0^2)^4 + (f_3 r)^4 h^2) (g_0^4 + g_1^4 h + g_3^4 h^3)\\
&= (f_1 g_0 + f_3 g_0^3)^4 + (f_1 g_1 + f_3 g_0^2 g_1 + f_3 g_3 r h)^4 h + (f_3 g_0 r)^4 h^2 + (f_1 g_3 + f_3 g_0^2 g_3
+ f_3 g_1 r)h^3.
\end{align*}
By writing $df/dh$ as $(df/dg)(dg/dh) = r^2 df/dg$, we compute that
\begin{align*}
\SY(f,g) &= \frac{f_1^2 f_3^2}{df/dg} dg = \frac{f_1^2 f_3^2 r^2}{r^2 (df/dg)} dg = \frac{f_1^2 f_3^2 r^4}{r^2 (df/dg)} dh \\
\SY(f,h) &= \frac{(f_1 g_1 + f_3 g_0^2 g_1 + f_3 g_3 r h)^2 (f_1 g_3 + f_3 g_0^2 g_3 + f_3 g_1 r)^2 + (f_3 g_0 r)^4}{r^2 (df/dg)} {dh}\\
\SY(f,g) + \SY(f,h) &= \frac{g_1^2 g_3^2 (f_1^2 + f_3^2 (g_0^4 + r^2 h))^2}{r^2 (df/dg)} {dh} \\
&= \frac{g_1^2 g_3^2}{dg/dh} {dh} = \SY(g,h)
\end{align*}
as desired.
\end{proof}

Putting everything together, we have the following statement which includes \cite[Propositions 2.7, 2.8, 2.9]{SY18}
(with the same proofs up to cosmetic differences).
\begin{theorem}[Sugiyama-Yasuda] \label{T:symbol}
The map $\SY: R_X \times R_X \to \Omega_{k(X)/k}$ has the following properties.
\begin{enumerate}
\item[(a)]
It is $\Gamma$-equivariant in each argument and symmetric in the two arguments.
\item[(b)]
For $f,g \in R_X$, $\SY(f,g) = 0$ if and only if $f,g$ belong to the same $\Gamma$-orbit.
\item[(c)]
For $f,g,h \in R_X$, $\SY(f,g) + \SY(g,h) + \SY(h,f) = 0$.
\end{enumerate}
\end{theorem}
\noindent {Note that, since $\mathrm{char}(k)=2$, a pairing is symmetric if and only if it is anti-symmetric.}
\begin{proof}
We first observe that part (b) is just a restatement of Corollary~\ref{C:symbol is zero}.
Given (b), Lemma~\ref{L:symbol is additive} implies that
\[
\SY(f,g) + \SY(g,f) = \SY(f,f) = 0;
\]
together with Lemma~\ref{L:symbol is invariant}, this implies (a).
Given (a), {the condition} (c) is equivalent to Lemma~\ref{L:symbol is additive}.
\end{proof}

\begin{remark} \label{R:torsor}
Suppose that $k$ is algebraically closed. As observed in \cite[Lemma~3.3]{SY18}, for fixed $g$ and $a$,
the equation $\SY(f,g) = \,da$ is quadratic in $f_1, f_2, f_3$, and so by Tsen's theorem has a nonzero solution.
Consequently, the morphism $\SY(-, g) \colon R_X/\Gamma \to \Omega_{k(X)/k}$ is surjective. It is also injective by Corollary~\ref{C:symbol is zero}, and hence we may upgrade Theorem~\ref{T:symbol} to assert that $\SY$ equips $R_X/\Gamma$ with the structure of a $\Omega_{k(X)/k}$-torsor.
We will return to this point in \S\ref{sec:conic bundle}.
\end{remark}

\begin{remark} \label{R:symbol}
Our use of the term \emph{symbol} in reference to $\SY$ is meant to suggest the possibility of a conceptual interpretation for this construction, e.g., in terms of algebraic $K$-theory. While we do not have such an interpretation in mind, it is natural to look for one using the following observation of Yuichiro Hoshi,
spelled out in more detail (and put into a geometric framework) in \cite{Hos19}.

In complex function theory, the \emph{Schwarzian derivative} of a pair $f,g$ is defined (as in \cite[Chapter~10]{Hil76}; see also \cite{OT09}) as
\[
\{f,g\} = \frac{d}{dg} \left( \frac{d^2f/dg^2}{df/dg} \right) - \frac{1}{2} \left(\frac{d^2f/dg^2}{df/dg} \right)^2.
\]
The associated quadratic differential $\{f, g\} (dg)^2$ is known to have strong algebraic properties:
it is antisymmetric in the two arguments and invariant under linear fractional transformations on either side,
vanishes if and only if $f$ is a linear fractional transformation of $g$, and satisfies
\[
\{f,h\}(dh)^2 = \{f,g\}(dg)^2 + \{g,h\}(dh)^2.
\]
If one divides by 2, reduces modulo 2, and takes the square root
(interpreting $\tfrac{1}{2} d^2f / dg^2$ appropriately in terms of the Cartier operator),
one recovers precisely the definition of $\SY(f,g)$
and the properties of $\SY(f,g)$ according to Theorem~\ref{T:symbol}.
\end{remark}

\section{Pseudotame and tame morphisms}
\label{sec:pseudotame morphisms}

Throughout \S\ref{sec:pseudotame morphisms}, assume that $k$ is perfect.
With the properties of the SY symbol in mind, we now introduce a key relaxation of the definition of a tame morphism: the notion of a \emph{pseudotame} morphism in the sense of Sugiyama--Yasuda. This can be thought of as the condition that a morphism is tame ``up to fourth powers'', or more precisely ``up to $\Gamma$-equivalence''.

\begin{defn} \label{D:pseudotame}
For $f \in R_X$ carrying $x$ to $y \in \PP^1_k$,
choose a uniformizer $t$ of $\PP^1_k$ at $y$.
Following \cite[Definition~2.1]{SY18}, we say that $f$ is \emph{pseudotame} at $x$
if there exists an element $h \in k(X)$ such that $v_x(f^*t + h^4)$ is odd.
This definition has the following properties.
\begin{itemize}
\item
If $f$ is tame at $x$, then evidently $f$ is pseudotame at $x$.
\item
If $u$ is a uniformizer of $X$ at $x$ and we write
$f^* t  = a_1 u + a_2 u^2 + \cdots$, then $f$ is pseudotame at $x$ if and only if the first index {$i\geq 1$ with $a_i \neq 0$ and $4 \!\not|\ i$ is an odd number}.

\item
In the previous point, the index $i$ does not depend on the choice of $t$
and is invariant under base extension.
Consequently, the pseudotame condition  at $x$ is independent of the choice of $t$,
and is invariant under base extension in the following sense: if
$k'$ is a field extension of $k$ and $x' \in X'$ lies over $x$, then $f$ is pseudotame at $x$ if and only if 
the induced map $f: X' \to \PP^1_{k'}$ is pseudotame at $x'$.
\end{itemize}
For $U \subseteq X$ an open subscheme, we say that $f$ is \emph{pseudotame} on $U$ if it is pseudotame at every point of $U$ ; for $U = X$, we simply say that $f$ is \emph{pseudotame}. 
\end{defn}

\begin{lemma} \label{L:pseudotame to uniformizer}
An element $f \in R_X$ is pseudotame at $x \in X^\circ$ if and only $\Gamma f$
contains an element which is a uniformizer at $x$. In particular, this property is $\Gamma$-invariant.
\end{lemma}
\begin{proof}
The pseudotame property is preserved by the operations
$f \mapsto a^4 f + b^4$ (as this has an obvious effect on the element $h$)
and $f \mapsto f^3$; consequently, it is $\Gamma$-invariant.
Since a uniformizer at $x$ is obviously pseudotame at $x$,
by Remark~\ref{R:f in terms of g transformed}
any element in its $\Gamma$-orbit is likewise.

Conversely, suppose that $f \in R_X$ is pseudotame. By translating by the element $h$ from Definition~\ref{D:pseudotame},
we obtain an element $f_1 \in \Gamma f$ with odd order at $x$. By taking either $f_1$ or $1/f_1$, we find an element $f_2 \in \Gamma f$ with order at $x$ congruent to 1 mod 4. By multiplying by a suitable fourth power, we find an element $f_3 \in \Gamma f$ with order 1 at $x$, as desired.
\end{proof}

Since pseudotameness is a $\Gamma$-invariant property, it is natural to ask how it is expressed in terms of symbols.
The answer, reproduced here from \cite[Theorem~2.10]{SY18}, turns out to be quite simple.
\begin{lemma} \label{L:pseudotame and regular}
For $x \in X$, suppose that $f, g \in R_X$ are such that $g$ is pseudotame at $x$.
Then $f$ is pseudotame at $x$ if and only if $\SY(f,g)$ is regular at $x$ {(meaning that it belongs to $(\Omega_{k(X),k})_x$)}.
\end{lemma}
\begin{proof}
By Theorem~\ref{T:symbol} and Lemma~\ref{L:pseudotame to uniformizer}, we are free to move $f$ and $g$ within their respective $\Gamma$-orbits
(as both sides of the desired equivalence are preserved).
In particular, we may assume that $g$ is in fact a uniformizer at $x$.

If $f$ is pseudotame at $x$, then we may also assume that it is a uniformizer at $x$;
then the terms $f_0, f_1, f_2, f_3$ in  \eqref{eq:f in terms of g} are regular at $x$,
and $f_1^2 + f_3^2 g$ is nonzero at $x$. It follows that $\SY(f,g)$ is regular at $x$.

Conversely, suppose that $f$ is not pseudotame at $x$. We may then assume that $v_x(f) = 2$,
in which case $v_x(f_1) \geq 1$, $v_x(f_2) = 0$, $v_x(f_3) \geq 0$. Then
\[
v_x(f_1 f_3 + f_2^2) = 0, \qquad v_x(f_1^2 + f_3^2 g) \geq 1,
\]
and so $\SY(f,g)$ {is} the differential of an element of $k(X)$ with a pole of order $2v_x(f_1^2 + f_3^2 g) - 1$ at $x$. This order being positive and odd, {$\SY(f,g)$ is not} regular at $x$.
\end{proof}

We conclude this section by showing that the existence of tame and pseudotame morphisms is intricately linked. 
We follow the proof of \cite[Theorem~4.1]{SY18}, with some minor adjustments to accommodate the case where $k$ is finite. (In exchange, we do not attempt to optimize the degree of the morphism.) 

\begin{lemma} \label{L:pseudotame to tame}
If $f \in R_X$ is pseudotame, then there exists $g \in \Gamma f$ which is tame;
in particular, $X$ admits a tame morphism to $\PP^1_k$ if and only if it admits a pseudotame morphism to $\PP^1_k$.
\end{lemma}
\begin{proof}
Fix a point $\infty \in X^\circ$ and let $R$ be the coordinate ring of the affine scheme $X \setminus \{\infty\}$. For $g \in R$ nonzero, define $\deg(g)$ as the pole order at $\infty$, or equivalently the minimal $n \in \mathbb{Z}_{\geq 0}$ such that {$g \in H^0(X, \mathcal{O}_X(n\infty))$}.

Write $f = h_0/h_1$ with $h_0, h_1 \in R$ and put $f_1 := h_1^4 f \in \Gamma f \cap R$. 
Since $f_1$ is pseudotame at $\infty$, so is $f_1^{2e_1+1}$ for any nonnegative integer $e_1$; by taking $e_1$ sufficiently large, we can ensure that there exists $h_2 \in R$ such that $\deg(f_2)$ is odd for $f_2 := 
f_1^{2e_1+1} + h_2^4$.
More precisely, if $\deg(f_1)$ is odd we may take $h_2 = 0$. Otherwise, $\deg(f_1)$ must be divisible by 4; by completing at infinity, we can find $r>0$ such that for every $e_1\geq 0$, there exists $\overline{h}_{2} \in \calO_X(\frac{1}{4}\deg(f_1^{2e_1+1})\infty)/\frakm_{X,\infty}^r$ for which 
\[
f_1^{2e_1+1}|_{r\infty} + \overline{h}_2^4 \in \calO_X(\deg(f^{2e_1+1}_1)\infty)/\frakm_{X,\infty}^r
\]
is of odd multiplicity. (We are using here the observation that for $f \in k((t))$, if we expand $f^{2e+1}$ and then compute the difference between the lowest degree of a nonzero term and the lowest odd degree of a nonzero term, the result is independent of $e>0$.) Now take $e_1 \gg 0$ so that $H^0(X, \calO_X(\frac{1}{4}\deg(f_1^{2e_1+1})\infty)) \to H^0(r\infty, \calO_X(\frac{1}{4}\deg(f_1^{2e_1+1})\infty))$ is surjective; then we can lift $\overline{h}_2$ to the sought-after $h_2 \in R$.

Let $Y$ be the union of the set-theoretic zero loci of $f_2$ and $df_2$ on $X \setminus \{\infty\}$.
For $e_2$ a sufficiently large positive integer, we have $(2e_2 + 1)\deg(f_2) > 12g-2$,
and moreover we can find $h_3 \in R$ such that
$(2e_2 + 1) \deg(f_2) > 4 \deg(h_3)$ and $f_3 := f_2^{2e_2+1} + h_3^4$ does not vanish anywhere on $Y$.
This can be ensured by enforcing explicit values of $h_3$ at points of $Y$ (the number of which is independent of $e_2$).
In particular, $\deg(f_3)$ is odd. Moreover, since the set-theoretic zero locus of $df_3$ is equal to $Y$, the zeroes of $f_3$ are all simple.

We will take $f_4 \in \Gamma f$ to have the form $f_3^3 + h_4^4$ for a suitable choice of $h_4 \in R$.
To find $h_4$, let $m_y$ denote the order of $df_3 = f_2^{2e_2}\,df_2$ at $y \in Y$, so that 
\[
\sum_{y \in Y} m_y \deg_k(y) = \deg(df_3) = \deg(f_3) + 2g-1.
\]
Since we are assuming that $k$ is perfect, 
$m_y$ is even; let $I$ be the ideal of $R$ consisting of elements which vanish at $y$ to order $\lfloor m_y/4 \rfloor + 1$ for each $y \in Y$. We may then choose $h_4 \in R$ such that $f_3^3 + h^4_4$ is tame at all $y \in Y$,
and this property holds also with $h_4$ replaced by any element of its congruence class modulo $I$.
In fact, by Riemann-Roch, we may find such an element $h_4$ with
\begin{align*}
\deg(h_4) &\leq 2g + \sum_y (\lfloor m_y/4 \rfloor +1) \deg_k(y)\\
& \leq 2g + \sum_y (m_y/2) \deg_k(y)  = 3g + \frac{1}{2}(\deg(f_3)  -1)
\end{align*}
and so $4 \deg(h_4) \leq 12g -2 + 2\deg(f_3)$. Since $\deg(f_3) > 12g-2$, we have
$4\deg(h_4) < 3 \deg(f_3)$ and so $f_4 := f_3^3 + h_4^4$ is tame at $\infty$.

To complete the argument, we must check that $f_4$ is tame at any $x \in X^\circ \setminus (Y \cup \{\infty\})$.
Since $df_4 = 3f_3^2 df_3$ and $df_3$ does not vanish at $x$, $f_4$ can only be ramified at $X$ if $f_3$ vanishes at $x$. In this case, $f_3$ has a simple zero at $x$, so $df_4$ has a double zero and both $f_3^3$ and $f_4$ have ramification index 3 at $x$ (this being unaffected by the addition of $h_1^4$).
This proves the claim.
\end{proof}

\begin{remark}
The method of proof of Lemma~\ref{L:pseudotame to tame} cannot ensure that
$g$ is \emph{triply ramified}, i.e., that its ramification indices are all equal to 1 or 3;
it may be possible to modify the argument to achieve triple ramification away from $\infty$,
but the method depends essentially on the large pole at $\infty$.

Whether or not an arbitrary curve admits a triply ramified morphism to $\PP^1_k$ remains an open problem even
in characteristic 0.
Over $\CC$, Fried--Klassen--Kopeliovich showed that the generic genus-1 curve admits
a triply ramified morphism \cite{FKK00} (see also \cite[Proposition~6.4]{Sch03a});
this was extended to genus-$g$ curves for any $g$ by Artebani--Pirola
\cite{AP05}.
Over an algebraically closed field of characteristic $p\neq 3$, Schr\"oer proved that the generic genus-1 curve admits a triply ramified morphism \cite[Corollary~5.3]{Sch03b}, and for $g \geq 2$ the set of points in the moduli space of genus-$g$ curve corresponding to curves admitting a triply ramified morphism has dimension at least $2g-3$ \cite[Corollary~5.2]{Sch03b}.
\end{remark}

\section{Conic bundles}
\label{sec:conic bundle}

We now expand on a point of \cite{SY18} to exhibit a geometric obstruction to the existence of a tame morphism. 
For the moment, we do not require $k$ to be perfect.

\begin{defn}
Following Raynaud \cite[\S 4]{Ray82}, we may identify the canonical theta characteristic on $X^{(1)}$ with the image $B$ of $\pi^{(1)}_* d: \pi^{(1)}_* \calO_X \to \pi^{(1)}_* \omega_{X/k}$ as follows.
Note that there is an exact sequence 
\[
0 \to \calO_{X^{(1)}} \to \pi^{(1)}_* \calO_X \to B \to 0
\]
of sheaves on $X^{(1)}$. To write down the autoduality on $B$, form the exact sequence
\[
0 \to B \to \pi^{(1)}_* \omega_X \stackrel{c}{\to} \omega_{X^{(1)}/k} \to 0
\]
where $c$ denotes the Cartier operator; since $p=2$, the pairing $(f,g) \mapsto c(f\,dg)$ 
on $\pi^{(1)}_* \calO_X$ induces the desired isomorphism
\begin{equation} \label{eq:autoduality}
B \otimes B \to \omega_{X^{(1)}/k}.
\end{equation}

By \eqref{eq:autoduality} plus Serre duality, we have a perfect pairing
\begin{equation} \label{eq:Serre duality}
H^0(X^{(1)}, B) \times H^1(X^{(1)}, B) \to k;
\end{equation}
{explicity, for $\phi \in H^0(X^{(1)}, B) \simeq \Hom(B,\omega_{X^{(1)}/k})$,
the corresponding map $H^1(X^{(1)}, B) \to k$ is given by}
\[
H^1(X^{(1)}, B) {\to} H^1(X^{(1)}, \omega_{X^{(1)}/k})  \xrightarrow{\mathrm{Tr}} k.
\]
Here, the trace map $\mathrm{Tr}$ is a sum of appropriate residues (cf.\ \cite[Section III.7]{Har77}).

To write this pairing out explicitly, fix an open affine covering $\{U,V\}$  of $X$
and suppose  $\alpha \in H^0(X^{(1)}, B)$ is represented by some $f \in R_X$ which is regular and not a square on $U$.
Any class $\beta \in H^1(X^{(1)}, B)$ can then be represented by a single value in $H^0(U^{(1)} \cap V^{(1)}, B)$,
which can in turn be represented by some $g \in k(X)/k(X)^2$. 
Writing $g = g_0^2 + g_1^2 f$, we have $c(fdg) = g_1df \in H^0(U \cap V, \omega_{X^{(1)}/k})$, and so
\begin{equation} \label{eq:pairing explicit}
\langle \alpha, \beta \rangle = \sum_{x \in X \setminus U} \Res_x(g_1\,df).
\end{equation}
(Note that we must break symmetry by summing over either $X \setminus U$ or $X \setminus V$; the choice does not matter because the sum over $X \setminus (U \cap V)$ vanishes by the residue theorem.)
\end{defn}

We now come to \cite[Theorem~3.5]{SY18}.
\begin{defn}\label{D:torsor}
Let $U_1,\dots,U_n$ be a covering of $X$ by open affine subspaces.
For $i=1,\dots,n$, let $f_i \in R_X$ be a morphism which is pseudotame on $U_i$.
For $i,j=1,\dots,n$, $\SY(f_i, f_j)$ is represented by a regular differential on $U_i \cap U_j$ by Lemma~\ref{L:pseudotame and regular},
and hence defines an element of $H^0(U_i^{(1)} \cap U_j^{(1)}, B)$. 
By Theorem~\ref{T:symbol}, these elements define a class in $H^1(X^{(1)}, B)$. Moreover, changing the choice of the $f_i$ (after possibly refining the covering) has the effect of translating
the cocycle by a class in $\bigoplus H^0(U_i^{(1)}, B)$, and so does not change the resulting class. Indeed, for $f_i'$ unramified on $U_i$, we have $\SY(f'_i,f_j) = \SY(f_i,f_j) + \SY(f'_i, f_i)$.

By Lemma~\ref{L:coboundary} below, the class in $H^1(X^{(1)}, B)$ vanishes, so we get a 1-coboundary on $X^{(1)}$ with values in $B$. The collection of coboundaries that occur in this fashion form a canonical $H^0(X^{(1)}, B)$-torsor, which we call the \emph{SY torsor} of $X$; we may view this torsor in a natural way as the set of $k$-points of an affine space $\calS_X$ over $k$.
(Recall that $X$ is ordinary if and only if $H^0(X^{(1)}, B) = 0$, in which case the SY torsor is reduced to a single point.)
\end{defn}

\begin{lemma} \label{L:coboundary}
The class in $H^1(X^{(1)}, B)$ defined in Definition~\ref{D:torsor} vanishes.
\end{lemma}
\begin{proof}
Since \eqref{eq:Serre duality} is a perfect pairing, it suffices to check that the specified class pairs to zero
with the class of $H^0(X^{(1)}, B)$ represented by an arbitrary element $g \in R_X$.
Let  $U$ be the open subspace of $X$ on which $g$ is regular and unramified;
we can then find an element $h \in k(X)$ such that $g + h^2$ is regular and tame on some open subspace $V$ of $X$
containing $X \setminus U$. Write $h = h_0^2 + h_1^2 g$.
We represent the specified class in $H^1(X^{(1)}, B)$ as the 1-cocycle with respect to the covering $\{U^{(1)},V^{(1)}\}$ taking the value
\[
\SY(g, g+h^2) = \SY(g, h_0^4 + g + h_1^4 g^2) = h_1^4\,dg
\]
on $U^{(1)} \cap V^{(1)}$. By \eqref{eq:pairing explicit}, the desired value of the pairing is
\[
\sum_{x \in X \setminus U} \Res_x(h_1^2\,dg) = \sum_{x \in X \setminus U} \Res_x(dh) = \sum_{x \in X \setminus U} 0 = 0,
\]
as claimed.
\end{proof}

We next  extract a key construction from the proof of  \cite[Lemma~3.3]{SY18}. Note that constructing an F-splitting $\pi^{(1)}_*\calO_X = \calO_{X^{(1)}} \oplus B^1$ over an open subset $U \subseteq X$ is equivalent to finding a global section $g$ of $\calO_U$ with the image in $\Gamma(U^{(1)}, B)$ vanishing nowhere on $U^{(1)}$. We denote this splitting as $\pi^{(1)}_*\calO_X = \calO_{X^{(1)}} \oplus \calO_{X^{(1)}}g$ over $U$. In turn, we obtain splittings of higher Frobenius pushforwards, for example, $\pi^{(2)}_* \calO_X = \bigoplus_{i=0}^3 \calO_{X^{(2)}} g^i$ over $U$.

\begin{defn} 
Form the $\PP^2$-bundle $Y := \PP( \pi_*^{(2)} \calO_X/\calO_{X^{(2)}})$ on $X^{(2)}$.
For $U$ an open subset of $X$, we can define homogeneous coordinates $T_1, T_2, T_3$ on $Y$ over $U^{(2)}$
by choosing $g \in \Gamma(U^{(2)}, \pi_*^{(2)} \calO_X) = \Gamma(U, \calO_X)$
whose image in $\Gamma(U^{(1)}, B)$ does not vanish anywhere,
and splitting $\pi^{(2)}_* \calO_X/\calO_{X^{(2)}}$ over $U$ as 
$\bigoplus_{i=1}^3 \calO_{X^{(2)}} g^i$.

Let $B^{(1)}$ be the pushforward of $B$
to $X^{(2)}$ via the relative Frobenius $\pi^{(1,2)}: X^{(1)} \to X^{(2)}$.
Given a section $a \in \Gamma(U^{(2)}, B^{(1)})$, we can write
$a = b^2 g$ with $b \in \Gamma(U^{(2)}, \pi^{(1,2)}_* \calO_{X^{(1)}})$
and then form the subscheme $C_{g,da}$ of $Y \times_{X^{(2)}} U^{(2)}$ 
cut out by $T_1 T_3 + T_2^2 + b (T_1^2 + g T_3^2)$;
this is a bundle of smooth conics over $U^{(2)}$.
Note that a $k(X^{(2)})$-valued point $[f_1: f_2: f_3]$ of $C_{g,da}$ corresponds to a solution
$f = f_0^4 + f_1^4 g + f_2^4 g^2 + f_3^4 g^3 \in k(X)$ of the equation
$\SY(f,g) = da$.

This construction is independent of coordinates in the following sense.
Suppose that $h \in \Gamma(U^{(2)}, \pi_*^{(2)} \calO_X)$ also has image in 
$\Gamma(U^{(1)}, B)$ which does not vanish anywhere. Then from Theorem~\ref{T:symbol} and the previous paragraph, we have an equality
\[
C_{g,da} = C_{h,da+ \SY(g,h)}
\] 
of closed subschemes of $Y \times_{X^{(2)}} U^{(2)}$. In order to describe this identity explicitly, let us denote the change of coordinates induced by the equality $\pi^{(2)}_* \calO_U/\calO_{U^{(2)}} = \bigoplus_{i=1}^3 \calO_{U^{(2)}} g^i = \bigoplus_{i=1}^3 \calO_{U^{(2)}} h^i$ by $(T_1,T_2,T_3) \mapsto (T'_1, T'_2, T'_3)$. By Theorem~\ref{T:symbol}, the equation $\SY(f,g)=da$ is equivalent to $\SY(f,h) = da + \SY(g,h)$, and so we get that
\[
\frac{T'_1T'_3 + (T'_2)^2}{(T'_1)^2 + h(T'_3)^2} = b'.
\]
over $k(X^{(2)})$ where $da + \SY(g,h) = (b')^2dh$. Since $X$ is integral, this is equivalent to $T'_1 T'_3 + (T'_2)^2 + b' ((T'_1)^2 + h (T'_3)^2)=0$ over $U^{(2)}$ which concludes the explanation.
\end{defn}

\begin{defn}
With notation as in Definition~\ref{D:torsor}, each element of the SY torsor is represented by a 0-cochain
$(a_i)_i$ with $a_i \in H^0(U_i^{(1)}, B)$. These elements satisfy
\[
\SY(f_i, f_j) = da_i + da_j.
\]
For each $i$, consider the conic bundle $C_{f_i, da_i}$ over $U_i^{(2)}$; by the previous discussion, 
the restrictions of $C_{f_i, da_i}$ and $C_{f_j,da_j}$ coincide as closed subschemes of
$Y \times_{X^{(2)}} (U_i \cap U_j)^{(2)}$. These bundles therefore glue to give a conic bundle over 
$X^{(2)}$ contained in $Y$; we call this the \emph{SY bundle} of $X$ associated to the original cochain. 

We may naturally globalize the construction over $\calS_X$ to obtain a conic bundle over $X^{(2)} \times_k \calS_X$
contained in $Y \times_k \calS_X$.
We call this the \emph{{total} SY bundle} of $X$, denoted hereafter by $Z_X$.
\end{defn}

\begin{lemma} \label{L:SY class to pseudotame}
Set notation as in Definition~\ref{D:torsor} and suppose that $k$ is perfect. Then 
there is a canonical bijection between pseudotame morphisms $f \in R_X$ and {pairs of dashed arrows which complete the commutative diagram
\[
\xymatrix{
X^{(2)} \ar@{-->}[r] \ar[d] & Z_X \ar[d] \ar[r] & X^{(2)} \ar[d] \\
\Spec(k) \ar@{-->}[r] & \calS_X \ar[r] & \Spec(k)
}
\]
in such a way that the horizontal compositions are identity morphisms.}
\end{lemma}
Since $X$ is a curve, giving a section over $X^{(2)}$ is equivalent to giving a section over the fraction field $k(X^{(2)})$.
\begin{proof}
A diagram as above corresponds to the choice of $k$-rational point of $\calS_X$ and a section of the fiber of $Z_X$ over this point. The first choice amounts to picking a 0-cochain $(a_i)_i$ with $a_i\in H^0(U_i^{(1)}, B)$; the second choice
amounts to picking a $k(X^{(2)})$-rational point of the fiber, which in turn corresponds to a solution $f \in k(X)$
of the system of equations $\SY(f, f_i) = da_i$. By Lemma~\ref{L:pseudotame and regular}, any such $f$
is pseudotame (and in particular belongs to $R_X$).

Conversely, if we start with $f$ pseudotame, then Lemma~\ref{L:pseudotame and regular} implies that $da_i := \SY(f,f_i)$ belongs to $\Gamma(U^{(1)}_i, B)$, so we may reverse the constructions of the previous paragraph. This proves the claim.
\end{proof}

\begin{remark} \label{R:double cover construction}
It is natural to wonder if one can describe the geometric structure of the conic bundles we constructed above. In this remark we show that they are in fact $2:1$ purely inseparable covers of $\mathbb{P}_{X}(\pi^{(2)}_*\calO_X/\pi^{(1)}_*\calO_X)$. For simplicity, we assume that $k$ is perfect.

Note that every conic in a two-dimensional projective space defined over an algebraically closed field of characteristic two is strange, that is, all tangent lines to it pass through a single point, called the \emph{strange point}. Moreover, they are all purely inseparable $2:1$ covers of $\mathbb{P}^1$ constructed by blowing up the strange point and sending a point $x$ on the conic to the intersection of the tangent line at $x$ with the exceptional locus.

To prove the claim, we first show that the section of $Y= \mathbb{P}_{X}(\pi^{(2)}_*\calO_X/\calO_X)$ given by the short exact sequence
\[
0 \to \pi^{(1)}_*\calO_X/\calO_X \to \pi^{(2)}_*\calO_X/\calO_X \to \pi^{(2)}_*\calO_X/\pi^{(1)}_*\calO_X \to 0
\] 
is a section of strange points. Indeed, our conic is given by the equation $T_1 T_3 + T_2^2 + b (T_1^2 + g T_3^2)=0$ and the above section picks up the point $(0:1:0)$. Since the derivative of this conic is $T_3dT_1 + T_1 dT_3$, we get that the tangent lines are given by equations of the form $f_3T_1 + f_1T_3=0$, and thus they all pass through the point $(0:1:0)$.

By globalizing the construction explained in the second paragraph, every conic bundle as above admits a $2:1$ purely inseparable map to the exceptional locus of the blow-up of it at the section of strange conics. Since this exceptional locus is isomorphic to $\mathbb{P}_{X}(\pi^{(2)}_*\calO_X/\pi^{(1)}_*\calO_X)$, the claim follows.

In fact, by Lemma \ref{lemma:relative-Frobenius} the map from our conic bundle to the exceptional locus of the blow-up is the relative Frobenius over $X$. This implies that the Frobenius twist of our conic bundle is isomorphic to $\mathbb{P}_{X}(\pi^{(2)}_*\calO_X/\pi^{(1)}_*\calO_X)$. \end{remark}

\begin{lemma} \label{lemma:relative-Frobenius} Let $S$ be a smooth scheme of finite type over a field $k$ of positive characteristic $p>0$. Let $X$ and $Y$ be smooth $S$-curves, and let $f \colon X \to Y$ be a purely inseparable finite morphism of degree $p$ over $S$. Then $Y$ is isomorphic to the Frobenius twist of $X$ over $S$ and $f$ is the relative Frobenius.
\end{lemma}
\begin{proof}
Since $f$ is purely inseparable of degree $p$, we have that $\mathcal{O}_X^{p} \subseteq \mathcal{O}_Y \subseteq \mathcal{O}_X$ (cf.\ \cite[Tag 0CNF]{SP}). Thus the absolute Frobenius $F_X$ factors through $f$, i.e.~$F_X \colon X \xrightarrow{f} Y \xrightarrow{g} X$.

Consider the following diagram
\begin{center}
\begin{tikzcd}
X \arrow[bend left = 30]{rr}{F_X} \arrow{r}{f} \arrow{rd} & Y \arrow{r}{g} \arrow{d} & X \arrow{d} \\
& S \arrow{r}{F_S} & S. 
\end{tikzcd}
\end{center}
Here the square is commutative because $\mathcal{O}_Y \subseteq \mathcal{O}_X$ and the big diagram is commutative. More precisely, we have that $X \xrightarrow{f} Y \xrightarrow{g} X \to S$ coincides with $X \xrightarrow{f} Y \to S \xrightarrow{F_S} S$, and hence so does $Y \xrightarrow{g} X \to S$ with $Y \to S \xrightarrow{F_S} S$ (as $f$ is an epimorphism).

In particular, we get a get a finite map $Y \to X'$, where $X' = X \times_S S$ is the Frobenius twist of $X$ over $S$. As $\deg f = p$, we must have that $\deg g = p^{n-1} = \deg F_S$ where $n$ is the dimension of $X$, and so $Y \to X'$ is of degree one. Hence $Y \simeq X'$ as $X'$ is normal.  By \cite[Tag 0CCY]{SP}, $f \colon X \to Y\simeq X'$ is the relative Frobenius over the generic point of $S$, and so $f$ is the relative Frobenius, as $X$ and $Y$ are integral.
\end{proof}

\section{The SY bundle and tame morphisms}
\label{sec:pseudotame}

Using the SY bundle, we establish our main results on the existence of tame morphisms.
We start by recovering the main result of Sugiyama--Yasuda \cite[Theorem~1.1]{SY18},
as well as a related observation \cite[Remark~3.7]{SY18}.
\begin{theorem}[Sugiyama--Yasuda] \label{T:torsor structure}
Suppose that $k$ is algebraically closed (and recall that $p=2$).
\begin{enumerate}
\item[(a)]
{Every SY bundle of $X$} admits a section. Consequently, by Lemma~\ref{L:pseudotame to tame} and Lemma~\ref{L:SY class to pseudotame}, $X$ admits a tame morphism to $\PP^1_k$.
\item[(b)]
{Via the symbol map $\SY$ of Definition~\ref{D:symbol}}, the subset of the quotient set $R_X/\Gamma$ consisting of orbits composed of pseudotame elements {is naturally identified with $\calS_X(k)$};
in particular, it carries the structure of a torsor for the group $H^0(X^{(1)}, B)$.
\end{enumerate}
\end{theorem}
\begin{proof}
To obtain (a), apply Tsen's theorem (Remark~\ref{R:torsor}).
To obtain (b),
note that for any pseudotame $f \in R_X$ and any $a \in H^0(X^{(1)}, B)$, Tsen's theorem again implies the existence of some $g \in R_X$ with $\SY(f,g) = da$;
Theorem~\ref{T:symbol} implies that $\Gamma g$ is uniquely determined by $f$ and $a$.
Lemma~\ref{L:pseudotame and regular} then implies that $g$ is pseudotame, and
Theorem~\ref{T:symbol} again provides the {identification with $\calS_X(k)$}.
\end{proof}

We note in passing the following open problem related to this result.

\begin{problem} \label{prob:degree of tame}
Suppose that $k$ is algebraically closed (and $p=2$). 
For each positive integer $g$, what is the smallest integer
$c(g)$ for which every curve $X$ of genus $g$ over $k$ admits a tame morphism  to $\PP^1_k$ of degree at most $c(g)$?
\end{problem}

\begin{remark}
Sugiyama--Yasuda show that in Problem~\ref{prob:degree of tame}, $c(g)$ exists and is at most $144g^2 + 66g-3$ \cite[Theorem~5.2]{SY18}. However, we expect the correct bound to be linear in $g$. 
\end{remark}

We next deduce Theorem~\ref{T:main1}.
\begin{defn}
For each element of the SY torsor of $X$, the corresponding SY bundle defines a class in $\Br(X^{(2)})[2]$,
which is trivial if and only if the bundle admits a section.
Each such class is called an \emph{SY class} of $X$.
\end{defn}

The following seems plausible, but we were unable to verify it. In addition to it being true when $k$ is finite or algebraically closed, see Theorem~\ref{T:supersingular elliptic} for another bit of corroborating evidence.
\begin{conj}
There is only one SY class associated to $X$; that is, the construction is
 independent of the choice of an element of the SY torsor.
\end{conj}

\begin{theorem} \label{T:finite to tame}
If $k$ is finite, then every SY class of $X$ is trivial. Consequently, by Lemma~\ref{L:pseudotame to tame} and Lemma~\ref{L:SY class to pseudotame}, $X$ admits a tame morphism to $\PP^1_k$.
\end{theorem}
\begin{proof}
Since $k$ is finite, class field theory implies that the Brauer group of $X$ is trivial. Indeed, by the Albert--Brauer--Hasse--Noether theorem, we have an injection
\[
\Br(K)[2] \to \bigoplus_{v \in X^{(2)}} \Br(K_v)[2],
\]
where $K = k(X^{(2)})$ and $K_v$ is the fraction field of the completion of the local ring $\calO_{X^{(2)},v}$ at $v$. When a conic bundle over a local field has good reduction, it is automatically trivial, as a section over the residue field can be lifted by means of Hensel's lemma. The theorem follows, as an SY bundle {by construction is a conic bundle over all of $X^{(2)}$}.%has everywhere good reduction (we have that $d(T_1 T_3 + T_2^2 + b (T_1^2 + g T_3^2)) = T_1dT_3 + T_3dT_1$ is non-zero at all residue fields of $X^{(2)}$).
\end{proof}

We will obtain Theorem~\ref{T:main2} by expanding on the same line of reasoning.

\begin{theorem} \label{T:extension to realize tame}
For $g_X > 1$, there exists a function $N = N(g_X)$ satisfying the following conditions.
\begin{enumerate}
\item[(a)]
Each SY class of $X$ is killed by some field extension $k'/k$ of degree at most $N$.
\item[(b)]
There exists a further field extension $k''/k$ of degree at most $N$ such that $X_{k''}$ admits a tame morphism to $\PP^1_{k''}$.
\end{enumerate}
\end{theorem}
\begin{proof}

First, we show that there exists a field extension $k'/k$, of degree depending only on $g_X$, for which $X(k') \neq \emptyset$.
Since we assumed $g_X > 1$, by Riemann-Roch $h^0(X, \omega_X)>0$, and so there exists a nonzero irreducible divisor $D$ such that $\deg D \leq 2g-2$. Taking $k'$ to be the {residue field of $D$} concludes the proof. 

We may assume hereafter that $X$ has a $k$-rational point.
Since $\Br(X^{(2)}_{\overline{k}})$ vanishes by Tsen's theorem, the Leray spectral sequence for the structure morphism $H^i(k, H^j_{\fppf}(X_{\overline{k}}, \GG_m)) \implies H^{i+j}_{\fppf}(X, \GG_m)$ yields an exact sequence
{
\[
\Br(k) \to \Br(X^{(2)}) \to H^1_{\et}(k, \underline{\Pic}(X^{(2)})) = H^1_{\fppf}(k,\underline{\Pic}(X^{(2)})).
\]
Since $X^{(2)}$ admits a rational point, a class in $\Br(k)$ which becomes 2-torsion in $\Br(X^{(2)})$ is automatically 2-torsion. Thus, we get a short exact sequence
 \[
\Br(k)[2] \to \Br(X^{(2)})[2] \to H^1_{\et}(k, \underline{\Pic}(X^{(2)}))[2].
\]
It suffices to bound the extension $k'/k$ required to kill the resulting class in $H^1_{\et}(k,\underline{\Pic}(X^{(2)}))[2]$,
as then we are only left to kill a class in $\Br(k')[2]$ with a further quadratic extension of $k'$
(this step is sometimes unnecessary; see Remark~\ref{R:last quadratic extension}).
Using the exact sequence
\begin{equation} \label{eq:Selmer sequence}
0 \to \Pic(X^{(2)})/ 2 \Pic(X^{(2)}) \to H^1_{\fppf}(k, \underline{\Pic}(X^{(2)})[2]) \to H^1_{\fppf}(k, \underline{\Pic}(X^{(2)}))[2] \to 0
\end{equation}
it further suffices to exhibit an extension $k'$ that kills a specified class in $H^1_{\fppf}(k, \underline{\Pic}(X^{(2)})[2])
= H^1_{\fppf}(k, \underline{\Pic}^0(X^{(2)})[2]) = H^1_{\fppf}(k, J[2])$
where $J$ is the Jacobian of $X^{(2)}$.}

Consider the connected-\'etale sequence
\[
1 \to J[2]^{\mathrm{conn}} \to J[2] \to J[2]^{\et} \to 1.
\]
To kill the image of the specified class in $H^1_{\et}(k, J[2]^{\et})$, it is enough to find a field extension $k'$ of $k$ for which $J[2]^{\et}$ admits a $k'$-rational point. If $x$ is any closed point of the scheme $J[2]^{\et}$, then we can take $k'$ to be the residue field $k(x)$ at $x$. Since the degree of this extension is bounded by the length of the zero-dimensional scheme $J[2]^{\et}$, which in turn depends only on the genus of $X$, the claim for the \'etale part is proven.

After replacing $k$ by $k'$, we can assume that our class lies in $H^1_{\fppf}(k, J[2]^{\mathrm{conn}})$. The morphism $J[2]^{\mathrm{conn}} \to \Spec k$ is a universal homeomorphism, and so it factors through the $r$-th power of Frobenius for some $r \in \mathbb{N}$ bounded in terms of the length of $J[2]^{\mathrm{conn}}$, and thus in terms of the genus. Taking a base change of $k$ by this $r$-th power makes the associated torsor acquire a rational point, and so kills our class in $H^1_{\fppf}(k, J[2]^{\mathrm{conn}})$.

This yields (a). To deduce (b), note first that if $k$ is perfect, then 
$X_{k'}$ itself admits a tame morphism by Lemma~\ref{L:pseudotame to tame} 
and Lemma~\ref{L:SY class to pseudotame}.
To handle the general case, note that the proof of  Lemma~\ref{L:pseudotame to tame}  still yields a function $f \in R_{X,k'}$ which becomes pseudotame after base extension to the perfect closure of $k'$;
then note that one may follow through the proof of Lemma~\ref{L:pseudotame to tame} over a purely inseparable extension of $k'$ whose degree can be bounded solely in terms of $g_X$.
\end{proof}

\begin{remark} \label{R:last quadratic extension}
If $k$ is perfect and $X$ has a $k$-rational point, then that gives a section of the map $\Br(k) \to \Br(X^{(2)})$.
Mapping any SY class back to $\Br(k)$ then gives a trivial class because our original conic bundle has a purely
inseparable multisection. Consequently, in this case it is not necessary to kill any classes in $\Br(k')[2]$.
\end{remark}

\section{Elliptic curves}
\label{sec:elliptic curves}

To illustrate the previous results in a concrete setting, we compute the obstruction to existence of a tame morphism for ordinary and supersingular elliptic curves in characteristic $2$.

Throughout \S\ref{sec:elliptic curves}, assume that $k$ is perfect, and let 
$\varphi: k \to k$ denote the absolute Frobenius morphism on $k$.

\begin{theorem} \label{T:elliptic curve obstruction}
Let $X$ be the elliptic curve over $k$ given by the affine model
\[
y^2 + xy = x^3 + ax^2 + b
\]
for some parameters $a,b \in k$ with $b \neq 0$.
(This is the generic form of an ordinary elliptic curve with $j$-invariant $b^{-1}$.)
\begin{enumerate}
\item[(a)]
If one of 
$a,b,a+b$ belongs to the image of $\varphi+1$, then $X$ admits a tame morphism. (Note that this always holds if $k$ is finite or algebraically closed.)
\item[(b)]
Conversely, if $X(k)$ is torsion and none of
$a,b,a+b$
belongs to the image of $\varphi+1$, then $X$ does not admit a tame morphism.
\end{enumerate}
\end{theorem}
\begin{proof}
For convenience, we write $A = a^{1/4}, B = b^{1/8}$, so that our curve becomes
\[
y^2 + xy = x^3 + A^4x^2 + B^8.
\]
Since $X$ is ordinary {and $k$ is perfect}, it has a unique $k$-rational 2-torsion point, namely $Q = (0, B^4)$. 
As per \cite[(4)]{Kra77}, the multiplication-by-2 map is given by
\[
(x,y) \mapsto \left(x^2 + \frac{b}{x^2}, (x + \frac{y}{x})(x^2 + \frac{b}{x^2}) + \frac{b}{x^2} \right).
\]
Consequently, the 4-torsion points have the form $(x,y)$ where
\[
x = B^2, \qquad y^2 + B^2 y = B^6 + A^4 B^4 + B^8.
\]
The latter equation is equivalent to
\[
a = \left( \frac{y}{B^2} + B^2 \right)^2 + \frac{y}{B^2} + B^2;
\]
from this,
we see that $X$ contains a rational 4-torsion point iff $A$ is in the image of $\varphi+1$ on $k$, where $\varphi$ denotes the absolute Frobenius (equivalently, if there is a change of coordinates to make $A$ equal to $0$).

Let $U$ be the open set on which $x$ is regular and pseudotame; note that $U = X \setminus \{\infty, Q\}$.
Let $V$ be the open set on which $x/(y + B^4)$ is regular and pseudotame; then
$U \cup V = X$ because $x/(y + B^4)$ is a uniformizer at both $Q$ and $\infty$.
For $g = x, f = (y+B^4)x^3  \in \Gamma(\frac{x}{y+B^4})$,
we calculate by repeated differentiation and using the equation of the elliptic curve that
\begin{align*}
f &= (xy + A^2 x^2 + B^4x)^2 + (x^2 + B^2 x)^2 x \\
&= (y + (A^2+A)x + B^4)^4 + x^4x + (x + B^2)^4 x^2 + B^4 x^3
\end{align*}
and hence
\[
\SY\left(\frac{x}{y+B^4}, x\right) = \SY(f,g) = \left( \frac{x^2 + Bx + B^4}{x(x+B^2)} \right)^2 dx.
\]
The fact that $y$ does not appear is not an accident: it follows from the fact that the nontrivial automorphism of the curve over $\PP^1_k$ is given by $y \mapsto y+x$, which carries $f = (y+B^4)x^3$ to the element $(y+x+B^4)x^3 = f + x^4$ in the same $\Gamma$-orbit.)
By writing
\[
\SY(f,g) = \frac{x^2+B^4}{x^2} dx + \frac{B^2 }{x^2+B^4} dx,
\]
we express $\SY(f,g)$ as a 1-coboundary (this expression is unique because $X$ is ordinary; moreover, $dx$ has only a double pole at $\infty$, so $\frac{B^2}{x^2+B^4}dx$ is regular at $\infty$).
The SY bundle therefore has generic fiber $C_{g,d(x+B^4x^{-1})}$, which is to say the zero locus of
\begin{equation} \label{eq:EC conic}
T_1 T_3 + T_2^2 + (1 + B^2 x^{-1})T_1^2 + (x + B^2)T_3^2.
\end{equation}
 Note that for $u,v \in k$, this conic admits a point of the form $(T_1: u T_1 + v T_3: T_3)$ if and only if
$(1+B^2 x^{-1} + u^2)(x + B^2 + v^2)$ is in the image of $\varphi+1$ on $k(X)$; here we use that for $w_1,w_2 \in k$, $w_1x^2 + xy + w_2y^2$ has a section if and only if $x^2 + x + w_1w_2$ does.
By making the choices $(u,v) = (0,0), (1,B), (1, x/B)$
and noting that
\[
(x^{-1} y)^2 + (x^{-1} y) + (B^4 x^{-1})^2 + (B^4 x^{-1}) = x + B^4 x^{-1} + A^4,
\]
we obtain a point in case one of
\[
A^4, B^2, A^4+B^2
\]
is in the image of $\varphi+1$ on $k$; this proves (a).

To prove (b), suppose by way of contradiction that $X(k)$ is torsion; none of $a,b,a+b$ belongs to the image of $\varphi+1$;
and the SY bundle $Z_X \to X$ admits a section. Extending the previous calculation, 
we see that any 8-torsion point $(x,y)$ of $X$ satisfies
\[
x^2 + \frac{b}{x^2} = B^2.
\]
Define the extensions
\[
k' = k[\alpha]/(\alpha^2 + \alpha + A^4), \qquad k'' = k'[\beta]/(\beta^2+\beta+B^2).
\]
By rewriting the equation of the curve as
\[
\left( \frac{y}{x} \right)^2 + \frac{y}{x} = x + a + \frac{b}{x^2}
= x^2 + x + a + x^2 + \frac{B^8}{x^2},
\]
we see that over $k''$, $X$ acquires the $8$-torsion point $(\beta B, \beta B (\beta B + \alpha + \beta))$; hence
$k''$ is the 8-division field of $X$, so we have a distinguished identification $\Gal(k''/k) \cong (\ZZ/8\ZZ)^\times$.

Since $Z_X \to X$ admits a section, $Z_X$ is a ruled surface and hence has the form
$\PP_X(\calE)$ for some rank-2 vector bundle $\calE$ over $X$ \cite[Proposition~V.2.2]{Har77} (this result assumes $k$ is algebraically closed, but it works for $k$ perfect as well).
From Remark~\ref{R:double cover construction}, we obtain an isomorphism
\[
\PP_X(F_* \calL_2) \cong \PP_X(F^* \calE)
\]
where the map $F: X \to X$ is Frobenius
and $\calL_2$ is the unique 2-torsion line bundle on $X$;
it follows that $F_* \calL_2 \cong F^* \calE \otimes \calG$ for some line bundle $\calG$.
Note that over $k'$, we have 
\[
F_* \calL_2 \cong (F_* \calO_X) \otimes \calL_4 \cong \calL_4 \oplus (\calL_2 \otimes \calL_4)
\]
where $\calL_4$ is a 4-torsion line bundle. 

We continue using the classification of elliptic surfaces on $X_{\overline{k}}$ from \cite[Theorem~2.15]{Har77}.
Suppose first that $\calE$ is decomposable on $X_{\overline{k}}$. 
By comparing $F^* \calE$ with $F_* \calL_2$ and using the decomposition
of the latter, we see that $\deg(\calE)$ is even.
Consequently, $\deg(\calG)$ is even; 
by twisting $\calE$ by a suitable multiple of $\calL(\infty)$, we may ensure that $\deg(\calG) = 0$.
Since $X(k)$ is torsion, there is a minimal positive integer $m$ such that $\calG^{\otimes m}$ is trivial.
If $m$ is odd, then $\calG$ has a square root with which we can twist $\calE$ to force $\calG$ to become trivial.
If $m = 2r$, then $r$ is odd because $X$ does not have 4-torsion over $k$. so $\calG \otimes \calL_2$ admits a square root
and so we may reduce to the case $\calG = \calL_2$. But $(F_* \calL_2) \otimes \calL_2 \cong F_* \calL$, so in this case we
may also take $\calG$ to be trivial.

Since $\calG$ is now trivial, on $X_{\overline k}$ we have
\[
F^* \calE \cong \calL_4 \oplus (\calL_2 \otimes \calL_4).
\]
On $X_{\overline{k}}$, we may choose a square root $\calL_8$ of $\calL_4$ and see that $\calE$ must split as one of
\[
\calL_8 \oplus (\calL_4 \otimes \calL_8), 
(\calL_2 \otimes \calL_8) \oplus (\calL_4 \otimes \calL_8), 
\calL_8 \oplus (\calL_2 \otimes \calL_4 \otimes \calL_8),
(\calL_2 \otimes \calL_8) \oplus (\calL_2 \otimes \calL_4 \otimes \calL_8).
\]
However, none of these options can be the pullback of a bundle on $X$: 
one of the elements of $\Gal(\overline{k}/k)$ fixes $\calL_4$ and sends $\calL_8$ to $\calL_2 \otimes \calL_8$,
and none of the candidates for $\calE$ is preserved by this element.
This yields the desired contradiction.

Suppose now that $\calE$ is indecomposable on $X_{\overline{k}}$. Then after twisting $\calE$, we must have a nonsplit extension
\[
0 \to \calO \to \calE \to \calO(n\infty) \to 0
\]
with $n \in \{0,1\}$.
This pulls back to an extension
\[
0 \to \calO \to F^*(\calE) \to \calO(2n\infty) \to 0.
\]
Since $F^*(\calE) \otimes \calG \simeq F_* \calL_2$, we have that $\calG(n\infty)$ squares to $\wedge^2 F_* \calL_2 \cong \calO$, and so it must be isomorphic to either $\calO$
or $\calL_2$. In either case, as $(F_*\calL_2) \otimes \calL_2 \simeq F_*\calL_2$, we get a sequence
\[
0 \to \calO \to (F_* \calL_2) \otimes \calO(n\infty) \to \calO(2n\infty) \to 0.
\]
If $n=0$, then we may split $F_* \calL_2$ as above and then obtain a contradiction; so we must have $n=1$.
Thus we get a sequence
\[
0 \to \calO \to \calO(P_1) \oplus \calO(P_2) \to \calO(2\infty) \to 0,
\]
where $P_1$,$P_2$ are the two $4$-torsion points. In particular, the map $\calO \to \calO(2\infty)$ yields a section in $H^0(X, \calO(2\infty))$ corresponding to $P_1+P_2 \in |2\infty|$. This is impossible, because this map is a Frobenius pullback of $\calO \to \calO(\infty)$.
\end{proof}

One can generate explicit examples using Kramer's description of 2-descent for elliptic curves over function fields of characteristic $2$ \cite{Kra77}. Take $k_0 = \FF_2(t)$; let $k$ be the perfect closure of $k_0$; let $X_0$ be the curve
$y^2 + xy = x^3 + ax^2 + b$ over $k_0$ for some $a \in k_0, b \in k_0 \setminus \FF_2$ such that
$b$ is a square in $k_0$, and none of $a,b,a+b$ belongs to the image of $\varphi+1$ on $k$;
and let $X$ be the base extension to $k$.
Let $\pi: X_0 \to X_0'$ be the Frobenius isogeny, whose target is the curve
$y^2 + xy = x^3 + a^2x^2 + b^2$; let $\psi: X_0' \to X_0$ be the dual isogeny (Verschiebung).
Then the product of the orders of the $\psi$-Selmer and $\pi$-Selmer groups gives an upper bound on 
the order of $X_0(k)/2X_0(k)$. 

Suppose in particular that the $\psi$-Selmer group has order 2 and the $\pi$-Selmer group is trivial. { Since the Mordell-Weil theorem holds in this case (see \cite[Theorem~1.1]{Ghi10}
or \cite[Corollary~1.3]{Ros15}, which apply because $j(X_0) = b^{-1} \notin \overline{\FF}_2$),}
 $X_0(k)$ is finitely generated and so must be equal to the subgroup $\ZZ/2\ZZ$ generated by $(0, b^{1/2})$; meanwhile, $\pi$ induces a surjection $X_0(k) \to X'_0(k)$ and so $X'_0$ is also of rank 0. It follows that $X(k)$ is torsion (and in fact is isomorphic to $\ZZ/2\ZZ$), and so $X$ is an example to which
Theorem~\ref{T:elliptic curve obstruction} applies; consequently, $X$ admits no tame morphism to $\PP^1_k$.

The following computation in the \textsc{Magma} computer algebra system confirms that taking $a = t, b = t^6$ yields an example as above. This may also be confirmed by a hand calculation.

\begin{verbatim}
> F<t> := RationalFunctionField(GF(2));
> E := EllipticCurve([1,t,0,0,t^6]);
> G1, G2 := TwoIsogenySelmerGroups(E);
> Order(G1), Order(G2);
2 1
\end{verbatim}

We summarize the result as follows.
\begin{theorem}
There exists an ordinary elliptic curve over a perfect field $k$ of characteristic $2$ (namely the perfect closure of $\FF_2(t)$) which does not admit a tame morphism to $\PP^1_k$. {Explicitly we can take the elliptic curve given by the equation $y^2 + xy = x^3 + tx^2+t^6$ over $\FF_2(t^{1/p^{\infty}})$.}
\end{theorem}

Moving to the supersingular case, we obtain the following.
\begin{theorem} \label{T:supersingular elliptic}
Let $X$ be the elliptic curve over $k$ given by the affine model
\[
y^2 + y = x^3 + ax + b
\]
for some parameters $a,b \in k$. (This is the generic form of a supersingular elliptic curve, necessarily having $j$-invariant $0$.) Then all of the SY classes of $X$ vanish; in particular, $X$ admits a tame morphism to $\PP^1_k$.
\end{theorem}
\begin{proof}
We have 
\[
dy = (x^2 + a)\,dx.
\]
Let $U$ be the complement of $\infty$ in $X$; then $x$ is regular and unramified on $U$.
Let $V$ be the open subset of $X$ on which $y/x^2$ is regular and unramified; then 
$U \cup V = X$.
For $g = x, f = x^2 y$, we calculate by repeated differentiation that
\begin{align*}
f &= (xy + b^{1/2} x)^2 + (x^2 + a^{1/2} x)^2 x  \\
&= (x^2 + a^{1/2} x)^4 + (x)^4 x + (y + b^{1/2} + b^{1/4})^4 x^2 + (a^{1/4})^4 x^3
\end{align*}
and hence
\[
\SY\left(\frac{y}{x^2}, x \right) = \SY(f,g) = \left( \frac{a^{1/4} x + y^2 + b^{1/2} + b}{x(x + a^{1/2})} \right)^2 dx.
\]
By writing
\[
\SY(f,g) = (x^2+a){\, dx} + \frac{a^{1/2} x^2 + y + x^3 + ax}{x^2(x^2+a)} dx,
\]
we express $\SY(f,g)$ as a 1-coboundary. However, since $X$ is not ordinary, this expression is not unique;
we have (here we use the fact that $H^0(X^{(1)}, B) \simeq H^0(X, \omega_X) = k\,dx$)
\[
H^0(X^{(1)}, B) = k x,
\]
and so we get a collection of SY bundles of the form $C_{g, d(x^3+c^2x)}$ for $c^2 \in k$. The corresponding conics are
\[
T_1 T_3 + T_2^2 + (x+c)T_1^2 + (x^2+cx) T_3^2
\]
or equivalently
\[
T_1 T_3 + T_2^2 + xT_1^2 + cx T_3^2;
\]
in the latter form, we may read off the rational point $(T_1:T_2:T_3) = (c^{1/2}:c^{1/4}:1)$.
\end{proof}

\begin{remark}
We did not attempt to explicitly reconstruct tame morphisms
in the setting of either Theorem~\ref{T:elliptic curve obstruction} or 
Theorem~\ref{T:supersingular elliptic}. However, in the latter case, we observe (following \cite[Remark~6.3]{Sch03a}) that when $a=0$, the rational function $y$ defines a (geometrically) cyclic covering of $\PP^1_k$, which is in particular tame.
\end{remark}

\section{Tame Belyi maps}

We now drop our running hypothesis that $p=2$ and formulate the refined tame Belyi theorem in positive characteristic. For parallelism, we first recall the usual Belyi theorem in characteristic zero, as in \cite[\S 2]{Gol12}.

\begin{theorem}[Riemann, Belyi]
Suppose that $p=0$.
\begin{enumerate}
\item[(a)]
If $X$ admits a morphism to $\PP^1_k$ ramified only over $\{0,1,\infty\}$, then $X$ descends to $\overline{\QQ}$.
\item[(b)]
Conversely, suppose that $k$ is algebraic over $\QQ$. Then $X$ admits a morphism to $\PP^1_k$ ramified only over $\{0,1,\infty\}$.
(In particular, this morphism is defined over $k$, not just over $\overline{k}$.)
\end{enumerate}
\end{theorem}

As noted in the introduction, the direct analogue of Belyi's theorem in positive characteristic fails without a tameness restriction, due to the following fact. (This statement has its origins in Abhyankar's observation that the map $x \mapsto x^p + x^{-1}$
presents $\GG_{m,k}$ as a finite \'etale cover of $\AAA^1_k$.)
\begin{theorem}[Abhyankar, Katz, et al.] \label{T:wild}
Suppose that $p>0$. Then $X$ always admits a morphism to $\PP^1_k$ ramified only over $\infty$.
\end{theorem}
\begin{proof}
See \cite[Lemma~16]{Kat88} for the case where $k$ is perfect,
and \cite[Theorem~1]{Ked05} for the general case (and a correspoding assertion in higher dimensions).
\end{proof}

To salvage a form of Belyi's theorem in positive characteristic, one must restrict to tame morphisms.
The following result recovers Theorem~\ref{T:tame Belyi}; part (b) was previously known for $k = \overline{\FF}_p$,
by Sa\"\i di for $p>2$ \cite[Th\'eor\`eme~5.6]{Sai97} and Anbar--Tutdere \cite[Theorem~2]{AT18} for $p=2$
(the latter using the work of Sugiyama--Yasuda).
\begin{theorem} \label{T:tame Belyi positive characteristic}
Suppose that $p>0$.
\begin{enumerate}
\item[(a)]
If $X$ admits a tame morphism to $\PP^1_k$ ramified only over $\{0,1,\infty\}$, then $X$ descends to $\overline{\FF}_p$.
\item[(b)]
Conversely, suppose that $k$ is algebraic over $\FF_p$. Then $X$ admits a tame morphism to $\PP^1_k$ ramified only over $\{0,1,\infty\}$.
\end{enumerate}
\end{theorem}
\begin{proof}
Part (a) follows from Grothendieck-Murre-Raynaud's theory of tame fundamental groups
(see in particular \cite[Expos\'e XIII, Corollaire~2.12]{Gro71}),
which implies that the tame fundamental group of a curve (and hence the category of tame covers) is invariant under base change along extensions of algebraically closed fields.
To prove (b), we may assume at once that $k$ is finite. Apply
Theorem~\ref{T:odd finite} (if $p>2$) or 
Theorem~\ref{T:finite to tame} (if $p=2$) 
to obtain a tame morphism $f_0: X \to \PP^1_k$.
Choose a power $q$ of $p$ such that all of the branch points of $f_0$ in $\PP^1_k$ are defined over $\FF_q$,
then let $f_1: \PP^1_k \to \PP^1_k$ be the map $x \mapsto x^{q-1}$. The composition $f_1 \circ f_0$ is tame and ramified only over $\{0,1,\infty\}$, as desired.
\end{proof}

\end{document}